\documentclass[11pt]{amsart}
\usepackage{tabularx,booktabs}
\usepackage{caption}
\usepackage{amsmath}
\usepackage{amsfonts}

\usepackage{amscd}
\usepackage{amsthm}
\usepackage{amssymb} \usepackage{latexsym}
\usepackage{eufrak}
\usepackage{euscript}
\usepackage{epsfig}
\usepackage{graphics}
\usepackage{array}
\usepackage{enumerate}
\usepackage{dsfont}
\usepackage{color}
\usepackage{wasysym}
\usepackage{hyperref}
\usepackage{pdfsync}
\usepackage{latexsym}
\usepackage{bbm}

\usepackage{graphicx}
\usepackage{float}
\usepackage{subfigure}

\newcommand{\bel}[1]{\begin{equation}\label{#1}}

\newcommand{\be}{\begin{equation}}

\newcommand{\ba}{\begin{eqnarray}}
\newcommand{\ea}{\end{eqnarray}}

\newcommand{\qe}{\end{equation}}
\newcommand{\R}{{\mathbb R}}
\newcommand{\N}{{\mathbb N}}
\newcommand{\Z}{{\mathbb Z}}
\newcommand{\C}{{\mathbb C}}

\newcommand{\Deg}{\mathrm{Deg}}

\newcommand{\Hmm}[1]{\leavevmode{\marginpar{\tiny%
$\hbox to 0mm{\hspace*{-0.5mm}$\leftarrow$\hss}%
\vcenter{\vrule depth 0.1mm height 0.1mm width \the\marginparwidth}%
\hbox to
0mm{\hss$\rightarrow$\hspace*{-0.5mm}}$\\\relax\raggedright #1}}}

\newtheorem{theorem}{Theorem}[section]

\newtheorem{lemma}[theorem]{Lemma}

\newtheorem{definition}[theorem]{Definition}
\newtheorem{conv}[theorem]{Assumption}
\newtheorem{remark}[theorem]{Remark}

\newtheorem{prop}[theorem]{Proposition}

\newtheorem{example}[theorem]{Example}

\begin{document}

\title[Uniqueness class to a class of linear evolution equations]{Uniqueness class of solutions to a class of linear evolution equations}

\author{Fengwen Han}
\address{Fengwen, Han: School of Mathematics and Statistics, Henan University, 475004 Kaifeng, Henan, China}
\email{\href{mailto:fwhan@outlook.com}{fwhan@outlook.com}}

\author{Bobo Hua}
\address{Bobo Hua: School of Mathematical Sciences, LMNS, Fudan University, Shanghai 200433, China; Shanghai Center for Mathematical Sciences, Fudan University, Shanghai 200433, China}
\email{\href{mailto:bobohua@fudan.edu.cn}{bobohua@fudan.edu.cn}}

\begin{abstract}
In this paper, we study the wave equation on infinite graphs. On one hand, in contrast to the wave equation on manifolds, we construct an example for the non-uniqueness for the Cauchy problem of the wave equation on graphs. On the other hand, we obtain a sharp uniqueness class for the solutions of the wave equation. The result follows from the time analyticity of the solutions to the wave equation in the uniqueness class. In the last part, we extend the result to a wide class of linear evolution equations.

\end{abstract}

\maketitle

\section{Introduction}\label{sec:intro}
Wave equations on Euclidean spaces, or Riemannian manifolds, play important roles in partial differential equations, Riemannian geometry, mathematical physics, etc.; see \cite{evans2010partial, hormander1997lectures, taylor2011partial, carroll2004an}. For a complete Riemannian manifold $M,$ the Cauchy problem of the wave equation reads as
\begin{align}\label{eq:wave1}
\left\{\begin{array}{lr}
\partial_{t}^{2} u(t, x)-\Delta u(t, x)=f(t, x), & (t, x) \in(-T, T) \times M, \\
u(0, x)=g(x), & x \in M, \\
\partial_{t} u(0, x)=h(x), & x \in M, \\
\end{array}\right.
\end{align} where $T\in(0,\infty],$ $\Delta$ is the Laplace-Beltrami operator on $M,$ $f,g$ and $h$ are appropriate functions. As is well-known, the strong solution, i.e. the $C^2$ solution, to the above equation \eqref{eq:wave1} is unique. 

Discrete analogs of partial differential equations on graphs, in particular elliptic and parabolic equations, have been extensively studied in recent years; see \cite{ grigor'yan2018introduction, barlow2017random}. In the first part of this paper, we study the wave equation on graphs. 

We recall the setting of weighted graphs. Let $(V,E)$ be a locally finite, connected, simple, undirected graph with the set of vertices $V$ and the set of edges $E$.  Two vertices $x,y$ are called neighbors if there is an edge connecting $x$ and $y$, i.e. $\{x,y\}\in E$, denoted by $x\sim y$. We denote by $d(x,y)$ the combinatorial distance between vertices $x$ and $y,$ i.e. the minimal number of edges in a path among all paths connecting $x$ and $y.$ Let
\[
\omega: E\rightarrow \mathbb (0,\infty),\ \{x,y\}\mapsto \omega_{xy}=\omega_{yx}
\] be the edge weight function,
\[
\mu:V\to (0,\infty),\ x\mapsto \mu_x
\]
be the vertex weight function. We call the quadruple $G=(V, E, \mu,\omega)$ a \emph{weighted graph} (or a network).

For a weighted graph $G=(V, E, \mu,\omega)$, the Laplacian of $G$ is defined as, for any function $u: V\rightarrow \mathbb R$,
\[
\Delta u(x):=\sum_{y\sim x}\frac{\omega_{xy}}{\mu_x}\left(u(y)-u(x)\right).
\] For any vertex $x$, the (weighted) degree of vertex $x$ is defined as 
\[
\Deg(x):=\sum_{y\sim x}\frac{\omega_{xy}}{\mu_x}.
\] 
The Laplacian is a bounded operator on $\ell^2(V,\mu),$ the $\ell^2$-summable functions w.r.t. the measure $\mu$, if and only if $\sup_{x\in V} \Deg(x)<\infty$; see \cite{keller2012dirichlet}. 

For any $\Omega\subset V,$ we denote by $$\delta\Omega:=\{y\in V\setminus \Omega: \exists x\in \Omega, y\sim x\}$$ the vertex boundary of $\Omega.$ We write $\overline{\Omega}:=\Omega\cup \delta\Omega.$

Let $I$ be an interval of $\R$. For $k\in\N_0\cup\{\infty\},$ we write $u\in C^k_t(I \times \Omega)$ if $u:I \times {\Omega}\to\R$ and $u(\cdot, x)\in C^k(I)$ for any $x\in \Omega.$ 

\begin{definition} For $\Omega\subset V$ and an interval $I$ of $\R$ containing $0,$ we say that $u:I \times \overline{\Omega}\to\R$ is a (strong) solution to the wave equation on $I \times \Omega$ (with Dirichlet boundary condition) if $u\in C^2_t(I \times \Omega)$ and $u$ satisfies 
\begin{align}\label{eq:wavegraph1}
\left\{\begin{array}{lr}
\partial_{t}^{2} u(t, x)-\Delta u(t, x)=f(t, x), & (t, x) \in I \times \Omega, \\
u(0, x)=g(x), & x \in \Omega, \\
\partial_{t} u(0, x)=h(x), & x \in \Omega,\\
u(t, x)=0, & (t, x) \in I \times \delta \Omega,
\end{array}\right. 
\end{align} where $f\in C^0_t(I \times \Omega),$ $g,h:\Omega\to\R.$ For $\Omega=V,$ the last condition is not required. For $f\equiv 0,$ it is called the solution to the homogeneous wave equation on $I \times \Omega.$
 \end{definition}

Lin and Xie \cite{lin2019the} proved the existence and uniqueness of solutions for the wave equation on a finite subset $\Omega$ of $V.$ In this paper, we consider the uniqueness problem for the solutions to the wave equation \eqref{eq:wavegraph1} on the whole graph.

For Riemannian manifolds, one of the fundamental properties for the solutions of the wave equation is
the finite propagation speed. By the finite propagation speed property, the solution to the wave equation \eqref{eq:wave1} is unique. For graphs, Friedman and Tillich \cite[p.249]{friedman2004wave} proved that the wave equation does not have the finite propagation speed property. 
In contrast to the uniqueness result for the wave equation on manifolds, we construct an example to show the non-uniqueness for the solutions to the Cauchy problem of the wave equation on graphs.
\begin{theorem}\label{thm:exam} (See Theorem \ref{sharp-example}) For the infinite line graph $\Z$ with unit weights, there is a nontrivial solution $u(t,x)\not\equiv 0$ to the wave equation \eqref{eq:wavegraph1} for $\Omega=\Z,$ $f,g,h\equiv 0.$\end{theorem}

Motivated by the above example, it is natural to ask the uniqueness class of the wave equation, i.e. the set of the solutions to the wave equation which possess the uniqueness property.
The uniqueness class of the heat equation on manifolds and graphs was extensively studied in the literature; see \cite{grigor'yan1986stochastically, grigor'yan1999analytic, huang2012uniqueness, huang2018uniqueness}.

For a weighted graph $G=(V, E, \mu,\omega),$ we fix a reference vertex $p\in V.$ In this paper, we always consider graphs satisfying the following assumption. 
\begin{conv}\label{def:ass} Let $G=(V, E, \mu,\omega)$ be a weighted graph satisfying that there exist constants $D>0,\ 0\leq \alpha\leq 2$ such that
\begin{equation}\label{deg-growth-cond}\Deg(x)\leq Dd(x,p)^\alpha, \ \ \forall x\in V, x\neq p.
\end{equation}
\end{conv} Note that this class of graphs includes many graphs with unbounded Laplacians. We introduce the following class of functions $\mathcal{M}_T$ for $T\in (0,\infty]$: $u(t,x)\in \mathcal{M}_T$ if $u:(-T, T) \times V\to \R$ and there exists a constant $C$ such that
\begin{equation}\label{eq:ass1}
|u(t,x)|\leq Cd(x,p)^{(2-\alpha) d(x,p)},\ \forall (t,x) \in (-T,T)\times V,\ x\neq p.
\end{equation}
We prove that $\mathcal{M}_T$ is a uniqueness class for the solutions of the wave equation on graphs.
\begin{theorem}\label{thm:main1} Let $G=(V,E,\mu,\omega)$ be a weighted graph satisfying Assumption~\ref{def:ass}. For $T\in (0,\infty],$ suppose that $u,v\in \mathcal{M}_T$ are solutions of the wave equation on $(-T,T)\times V$ with same data $f,g,h.$ Then $u\equiv v.$ \end{theorem}
\begin{remark}
The uniqueness class $\mathcal{M}_T$ is sharp in the sense that the exponent $2-\alpha$ in \eqref{eq:ass1} cannot be improved; see Theorem~\ref{sharp-example}.
\end{remark}

To prove the result, we first prove the time analyticity of the solutions to the wave equation on graphs. Then the theorem follows from the unique continuation property of analytic functions. 
We recall the results about the analyticity of the solutions of the wave equation on Riemannian manifolds. By the Cauchy-Kowalewski theorem \cite{kovalevskaja1874theorie}, suppose that the manifold $M$ is analytic and the data $f,g,h$ are analytic, then the solution of the wave equation \eqref{eq:wave1} is analytic in time, also in space. Recently, under some optimal growth condition, Dong and Zhang \cite{DongZhang19} proved the time analyticity of ancient solutions to the heat equation; see also Zhang \cite{zhang2020note}. By the discrete nature of graphs, one may guess that the solutions to the wave equation on graphs are always time-analytic. This is true for the solutions with Dirichlet boundary condition on a finite subset $\Omega\subset V$; see Proposition \ref{prop:finite}. However, in general this fails for infinite graphs. The same example as in Theorem~\ref{thm:exam} provides a counterexample for the time analyticity of the solutions to the wave equation; see Section~\ref{sec:sharp}. In the following, we prove the time analyticity of the solutions to the wave equation in the class $\mathcal{M}_T.$

\begin{theorem}\label{main-result}
Let $G=(V,E,\mu,\omega)$ be a weighted graph satisfying Assumption~\ref{def:ass}. Let $u$ be a solution of the homogeneous wave equation on $[0,T)\times V$ for some $T\in (0,\infty].$ Suppose that $u$ satisfies, for some $C>0$ and $A_1\in [0,2-\alpha],$
\begin{align}\label{func-growth-cond}
|u(t,x)|\leq C d(x,p)^{A_1 d(x,p)},\ \forall (t,x) \in [0,T) \times V,\ x\neq p.
\end{align}
Then $u=u(t,x)$ is analytic in $t\in [0,T)$ with analytic radius $r$ satisfying $r=+\infty$ (resp. $r\geq\frac{1}{e}\sqrt \frac{2}{D}$) if $A_1<2-\alpha$ (resp. $A_1=2-\alpha$).
\end{theorem}
\begin{remark} \begin{enumerate}[(i)]\item An analogous result for the heat equation on graphs was proved in \cite{han2019time}. Note that the above result holds for the half interval $[0,T),$ which is stronger than that for $(-T,T)$. For the heat equation on manifolds, these are different: the time analyticity fails for $[0,T)$ due to the counterexample of Kovalevskaya, but it is true for $(-T,0]$ by \cite{DongZhang19}.
\item For the conclusion, the condition that $A_1\leq 2-\alpha$ is sharp; see Theorem~\ref{sharp-example}.
\end{enumerate}
\end{remark}
To sketch the proof strategy, for simplicity, we consider the case that the graph $G$ has bounded degree, and let $u$ be the solution to the homogeneous wave equation as in Theorem~\ref{main-result}. Note that the Laplace operator $\Delta$ is a bounded operator on bounded functions, and in fact the supremum norm of $\Delta u$ on $B_R(p),$ the ball of radius $R$ centered at $p$, is bounded above by that of $u$ on $B_{R+1}(p)$; see Lemma~\ref{lemma:lap-est}. By the induction, we estimate $\Delta^k u$ on $B_R(p)$ for any $k\geq 1$ by the supremum norm of $u$ on $B_{R+k}(p),$ which is bounded above by the growth condition \eqref{func-growth-cond}. Using the  homogeneous wave equation, we get the bound for $\partial_t^{2k} u$ for any $k\geq1,$ which yields the time analyticity of $u$ via the remainder estimate for the Taylor series in time of $u.$ 

Our approach applies for a class of linear evolution equations of the following form:
$$\partial_t^m u(t,x)-Lu(t,x)=f(t,x), \qquad (t,x)\in [-T,T]\times V,$$
where $m$ is a positive integer, $V$ is a discrete set and $L$ is in a class of linear operators on $\R^V$, including the Laplacian $\Delta$, the Schr\"odinger operator $\Delta + W(x)$, and the biharmonic operator $\Delta^2$ on weighted graphs with even complex-valued weights; see Section \ref{sec:derive}.

The paper is organized as follows: In next section, we recall some basic facts on weighted graphs. In Section \ref{sec:sharp}, we construct the counterexample for both the uniqueness and the time analyticity for the solutions of the wave equation. Section \ref{sec:proof} is devoted to the proofs of main results, Theorem~\ref{thm:main1} and Theorem~\ref{main-result}. In Section \ref{sec:derive}, we extend the result to a class of linear evolution equations on discrete sets.

\section{Preliminaries}
We recall some facts on weighted graphs. Let $G=(V, E, \mu,\omega)$ be a weighted graph. We denote by $$B_R(x):=\{y\in V: d(x,y)\leq R\}$$ the ball of radius $R$ centered at $x$. For any subset $K$ in $V$, we write $$B_R(K):=\{y\in V: \exists x\in K,\ \mathrm{s.t.}\ d(x,y)\leq R\}.$$

The following lemma states that the bound of $\Delta^k f(x)$ is controlled by the bound of the function $f(x)$ on a graph. This is an useful estimate in proving the main result.
  
\begin{lemma}\label{lemma:lap-est}
Let $G=(V, E, \mu,\omega)$ be a weighted graph and K be a subset of $V$. Then
\begin{enumerate}[(i)]
\item\label{prop1} $|\Delta f(x)|\leq 2\Deg(x) \sup\limits_{y\in B_1(x)}|f(y)|,$
\item\label{prop2} $\sup\limits_{x\in K}|\Delta f(x)|\leq 2\sup\limits_{y\in K}\Deg(y)\sup\limits_{z\in B_1(K)}|f(z)|,$
\item\label{prop3} $|\Delta^k f(x)|\leq \left( 2\sup\limits_{y \in B_k (x)}\Deg (y)\right)^k\sup\limits_{z\in B_k (x)}|f(z)|,\ \mathrm{for}\ k \in \N.$
\end{enumerate}
\end{lemma}

\begin{proof}
The results (\ref{prop1}) and (\ref{prop2}) are obtained directly from the definition of Laplacian on graphs. For (\ref{prop3}), by using (\ref{prop1}) and (\ref{prop2}), we have
\begin{align*}
|\Delta^k f(x)|
\leq&2\Deg(x) \sup\limits_{y\in B_1(x)}|\Delta^{k-1} f(x)|
\leq \cdots \ \cdots \\
\leq&\left(2\sup\limits_{y \in B_k (x)}\Deg (y)\right)^k\sup\limits_{z\in B_k (x)}|f(z)|.
\end{align*}

\end{proof}

In the following, we show the time analyticity for the solutions of the wave equation on finite subsets with Dirichlet boundary condition.
\begin{prop}\label{prop:finite} Let $\Omega$ be a finite subset of $V,$ and $u$ be a solution to the homogeneous wave equation on $\R\times \Omega$, 
\begin{align}\label{eq:wavegraphhomo}
\left\{\begin{array}{lr}
\partial_{t}^{2} u(t, x)-\Delta u(t, x)=0, & (t, x) \in \R \times \Omega, \\
u(0, x)=g(x), & x \in \Omega, \\
\partial_{t} u(0, x)=h(x), & x \in \Omega,\\
u(t, x)=0, & (t, x) \in \R \times \delta \Omega.
\end{array}\right. 
\end{align}
Then
\begin{equation}\label{eq:d1}u(t,x)=\sum_{i=1}^N \cos(t\sqrt{\lambda_i})a_i\psi_i(x)+\sum_{i=1}^N \frac{1}{\sqrt{\lambda_i}}\sin(t\sqrt{\lambda_i})b_i\psi_i(x),\end{equation} 
where $ N = \sharp \Omega$, $0<\lambda_1 \leq \lambda_2 \leq \cdots \leq \lambda_N$ are the eigenvalues of the Laplace operator with Dirichlet boundary condition on $\Omega$, and $\psi_i$ are the 
corresponding orthonormal eigenvectors. Here $a_i$ and $b_i$ satisfy 
\[
 \sum_{i=1}^N a_i\psi_i(x)=g(x),\ \sum_{i=1}^N b_i\psi_i(x)=h(x).
\]
In particular, $u$ is analytic in time.
\end{prop}
\begin{proof} One can show that $u(t,x)$ given in \eqref{eq:d1} is a solution to the homogeneous wave equation on $\R\times \Omega$. Moreover, by the uniqueness result of \cite{lin2019the}, this is the solution. The solution is analytic in time by the expression of $u$ in (\ref{eq:d1}).
\end{proof}

\section{A non-uniqueness result}\label{sec:sharp}

In this section, we construct a nontrivial solution of the wave equation with zero initial data on the graph. Let $\mathbb Z = (V,E,\mu,\omega)$ be the infinite line graph with unit weights. Here $V = \mathbb Z$ is the set of integers. For all $x,y \in V$, $\{x,y\} \in E \Leftrightarrow |x-y|=1$. The weight functions are $\mu_x=1,\ \forall x \in V$ and $\omega_{xy}=1,\ \forall \{x,y\} \in E$. 

\begin{figure}[H]
\centering
\includegraphics[width=0.75\textwidth]{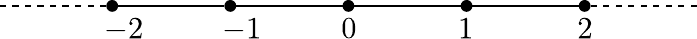}
\caption{Graph $\Z$, $\Deg(x)\equiv 2.$}
\label{fig.z}
\end{figure}

Following Tychonoff \cite{tychonoff1935}, Huang \cite[Section 3]{huang2012uniqueness} constructed a nontrivial solution of the heat equation with zero initial data on the infinite line graph $\mathbb Z$. By modifying his construction, we obtain a solution of the wave equation with zero initial data on $\mathbb Z$.

\begin{theorem}\label{sharp-example}
Let $\mathbb Z$ be the infinite line graph with unit weights. For any $\epsilon>0$, there exists a solution $u(t,x)$ to the homogeneous wave equation on $\R\times \Z$, which is not time-analytic, such that
\begin{equation}\label{eq:sharp}
\lim_{x \to \infty}|u(t,x)|e^{-(2+\epsilon)|x|\ln |x|}=0.
\end{equation}
 Moreover, $$u(0,x)=\partial_t u(0,x)=0, \ \forall x\in \Z.$$
\end{theorem}

\begin{remark}
The result of Theorem \ref{thm:exam} is a part of the above theorem.
\end{remark}

\begin{proof}[Proof of Theorem \ref{sharp-example}]
Set
\begin{align*}
g(t)=
\begin{cases}
\exp(-t^{-\beta}),\ \ &t>0,\\
0, & t\leq 0,
\end{cases}
\end{align*}
where $\beta$ is a positive constant. Then the following properties of  $g(t)$ hold:

\begin{enumerate}[(i)]
\item\label{g-prop-1} $0 \leq g(t) \leq 1,\ t \in \R$,
\item\label{g-prop-2} $g^{(k)}(0)=0$ for all $k \in \N_0$,
\item\label{g-prop-3}$|g^{(k)}(t)|\leq k!\left(\frac{2k}{e\beta\theta^\beta}\right)^\frac{k}{\beta}$, where $\theta$ is a positive constant depends on $\beta$.
\end{enumerate}
Here (\ref{g-prop-1}) and (\ref{g-prop-2}) follow from direct calculation. The statement (\ref{g-prop-3}) was proved by Huang \cite[Section 3]{huang2012uniqueness} and Fritz John \cite{john1991partial}.

We define a function $u(t,x)$ on $\R \times \mathbb Z$ as follows,
\begin{align*}\begin{split}
u(t,x)=
\begin{cases}
g(t), \ \ \ \ \ \ \ \ \ \ \ \ \ \ \ \ & x=0,\\
g(t)+\sum_{k=1}^\infty\frac{g^{(2k)}(t)}{(2k)!}(x+k)\cdots (x+1)x\cdots (x-k+1),  & x\geq 1,\\
u(-x-1,t),  &x\leq -1.
\end{cases}
\end{split}\end{align*}
Note that for any fixed $x \in \mathbb Z$ and $k>x$, $$\frac{g^{(2k)}(t)}{(2k)!}(x+k)\cdots (x+1)x\cdots (x-k+1)=0.$$
So that $u(t,x)$ is well defined on $\R \times \mathbb Z$. Moreover, $\partial_t^k u(0,x)=0$ holds for all $k \in \N_0$. One readily verifies that $u(t,x)$ solves the wave equation and $u(t,0)$ is not analytic.

Next we prove (\ref{eq:sharp}). Note that $u(t,x)$ is symmetric with respect to $x=-\frac{1}{2}$. It suffices to prove the result for $x \ge 0$. 

By the properties of the function $g(t)$, we have 
\begin{equation}\nonumber\begin{aligned}
|u(t,x)|\leq& 1+\sum_{k=1}^\infty (2k)!\left(\frac{4k}{e\beta\theta^\beta}\right)^\frac{2k}{\beta}\frac{1}{(2k)!}(x+k)\cdots (x+1)x\cdots (x-k+1)\\
=&1+\sum_{k=1}^\infty(x+k)\cdots (x+1)x\cdots (x-k+1)\left(\frac{4k}{e\beta\theta^\beta}\right)^{\frac{2k}{\beta}}\\
=:&1+\sum_{k=1}^\infty a_k ,
\end{aligned}\end{equation}
where $a_k=(x+k)\cdots (x+1)x\cdots (x-k+1)\left(\frac{4k}{e\beta\theta^\beta}\right)^{\frac{2k}{\beta}}$. Since $a_{k}\leq a_x$ for $0 \leq k \leq x$ and $\lim\limits_{x \to \infty}a_x=+\infty$, there exists a positive constant $X_0$ such that for any $x>X_0$ the following holds:
\begin{align*}
|u(t,x)|
\leq& 1+\sum_{k=1}^\infty a_k
=1+\sum_{k=1}^x a_k\\
\leq& 1+xa_x
\leq 2xa_x\\
=& 2x(2x)!\left(\frac{4x}{e\beta\theta^\beta}\right)^{\frac{2x}{\beta}}\\
=& (2x)!\exp\left(\frac{2x}{\beta}\ln x+O(x)\right),
\end{align*}
where $O(x)$ denotes a function satisfying $\varlimsup\limits_{x \to +\infty }{\frac {|O(x)|}{x}}<+\infty.$ By Stirling's formula, $x! \sim \sqrt{2\pi x}\left(\frac{x}{e}\right)^x,$ $x \to +\infty.$ There exists a positive constant $X_1$ such that for any $x>X_1,$
\[
(2x)! \leq 4 \sqrt{\pi x}\left(\frac{2x}{e}\right)^{2x}=\exp[2x\ln x+O(x)].
\]
Hence
\begin{equation}\nonumber\begin{aligned}
|u(t,x)| \leq \exp\left[\left(2+\frac{2}{\beta}\right)x\ln x+O(x)\right].
\end{aligned}\end{equation}
For any $\epsilon>0$, let $v_\epsilon(x)=\exp[(2+\epsilon)x\ln x]$. We have
\begin{align*}
\frac{|u(t,x)|}{v_\epsilon(x)} \leq \exp\left[\left(\frac{2}{\beta}-\epsilon\right)x\ln x + O(x)\right].
\end{align*}
For any $\beta > \frac{2}{\epsilon},$ we have $$\lim_{x \to \infty}\frac{|u(t,x)|}{v_\epsilon(x)} = 0.$$
This proves the result.
\end{proof}

By the same argument, we can construct examples on $\mathbb Z$ for higher order equations of type $\partial_t^m u(t,x)=\Delta u, m \geq 3,$ as follows:
\begin{align}\label{sharp-example-higher-order}
\begin{split}
u(t,x)=
\begin{cases}
g(t), \ \ \ \ \ \ \ \ \ \ \ \ \ \ \ \ & x=0,\\
g(t)+\sum_{k=1}^\infty\frac{g^{(mk)}(t)}{(2k)!}(x+k)\cdots (x+1)x\cdots (x-k+1),  & x\geq 1,\\
u(-x-1,t),  &x\leq -1,
\end{cases}
\end{split}
\end{align}
where $g(t)$ is defined as in the proof of Theorem \ref{sharp-example}.\\

\section{Proof of The Main Results}\label{sec:proof}

In this section, we prove the main results. The following lemma is well known in the calculus.

\begin{lemma}\label{lemma:dev-est}
For any $f(x) \in C^2([a,b])$, let $M_k=\max\limits_{x\in [a,b]}\{|f^{(k)}(x)|\}$, $k=0,1,2$. Then $M_1\leq \frac{2}{b-a}M_0+(b-a)M_2$. 
\end{lemma}

\begin{proof}[Proof]
By Taylor's formula, for any $x_0 \in [a,b]$ and $h \in [a-x_0,b-x_0]$, there exists $x_1 \in [a,b]$ such that

\begin{align*}
f(x_0+h)=f(x_0)+hf'(x_0)+\frac{{f}''(x_1)}{2}h^2.
\end{align*}
Hence we have
\begin{align*}
|f'(x_0)h+f(x_0)|
=&|f(x_0+h)-\frac{{f}''(x_1)}{2}h^2|\\
\leq & M_0+\frac{M_2(b-a)^2}{2}. 
\end{align*}
Set $g(h):=f'(x_0)h+f(x_0)$. Note that $g(h)$ is a linear function defined on $[a-x_0,b-x_0]$, and is bounded by $ M_0+\frac{M_2(b-a)^2}{2}$. So that the slope of $g(h)$ satisfies
\begin{align*}
|f'(x_0)|
\leq&2\left(M_0+\frac{M_2(b-a)^2}{2}\right)\frac{1}{(b-x_0)-(a-x_0)}\\
=&\frac{2}{b-a}M_0+(b-a)M_2.
\end{align*}
\end{proof}

This is a special case of Ore's inequalities on functions with bounded derivatives \cite{ore1938functions}.

\begin{theorem}[\cite{ore1938functions}]\label{ore-dev-est}
Let $f(x)\in C^{n+1}([a,b])$ satisfies
\begin{align*}
|f(x)|\leq M_0,\ \ \ |f^{(n+1)}(x)|\leq M_{n+1},\ \ \forall x \in [a,b].
\end{align*}
Then all intermediate derivatives of $f(x)$ are bounded above by
\begin{align*}
|f^{(i)}(x)|\leq\frac{K(i,n)}{(b-a)^i}\left(M_0+\frac{(b-a)^{n+1}}{(n+1)!}M_{n+1}\right),
\end{align*}
where $i=1,2,\cdots n$ and $K(i,n)=\frac{2^i n^2(n^2-1)\cdots(n^2-(i-1)^2)}{1\cdot 3\cdot 5\cdots (2i-1)}$.
\end{theorem}

Now we are ready to prove Theorem \ref{main-result}.
\begin{proof}[Proof of Theorem \ref{main-result}]
Note that $u$ is a strong solution, i.e. $u\in C^2_t([0,T)\times V),$ to the homogeneous wave equation $\partial_t^2 u = \Delta u.$ By the same argument for the time regularity in \cite{hua2015time}, we can prove that $u\in C^\infty_t([0,T)\times V).$ Since $\Delta$ commutes with $\partial_t$, we have
\[
\partial_t^{2k} u = \Delta^k u,\ \forall k \in \N.
\]
Fix $x_0\in V,\ t_0 \in [0,T]$, let $R_n=\frac{\partial_t^{n+1}u(s,x_0)}{(n+1)!}(t-t_0)^{n+1}$ be the $n$-th Lagrange remainder of $u(t,x_0)$ at $t_0$, where $s \in (t_0,t)$ or $(t,t_0)$. It suffices to prove that $$\lim_{n \to \infty} R_n =0.$$

Case I: If n is odd, i.e. $n=2k-1$, $k\in \N$, then
\begin{align*}
R_n=R_{2k-1}=\frac{\partial_t^{2k}u(s,x_0)}{(2k)!}(t-t_0)^{2k}=\frac{\Delta^k u(s,x_0)}{(2k)!}(t-t_0)^{2k}.
\end{align*}

Set $d:=d(x_0,p)$. This yields that $B_k(x_0)\subset B_{k+d}(p)$. By Lemma \ref{lemma:lap-est}, we have
\begin{align*}
|\Delta^k u(s,x_0)|
&\leq 2^k\left(\sup\limits_{y\in B_k(x_0)}\Deg(y)\right)^k\sup\limits_{z\in B_k(x_0)}|u(z)|\\
&\leq 2^k\left(\sup\limits_{y\in B_{k+d}(p)}\Deg(y)\right)^k\sup\limits_{z\in B_{k+d}(p)}|u(z)|.\\
\end{align*}
Combining this with (\ref{func-growth-cond}) and Assumption \ref{def:ass}, we have
\begin{equation}\label{k-th-lap-est}\begin{aligned}
|\Delta^k u(s,x_0)|
&\leq 2^k D^k \exp[\alpha k \ln(k+d)]\cdot C \exp[A_1(k+d)\ln (k+d)]\\
&=\exp[(A_1+\alpha)k\ln(k+d)+k\ln(2D)+A_1d\ln(k+d)+\ln C]\\
&=\exp[(A_1+\alpha)k\ln k+k\ln(2D)+o(k)],\\
\end{aligned}\end{equation}
where $o(k)$ denotes a function satisfying $\lim\limits_{k \to +\infty}{\frac{o(k)}{k}}=0$.

By Stirling's formula, there is a positive constant $K_1$ such that for any $k>K_1,$
\be\label{eq:stir-app}
(2k)!\geq \left(\frac{2k}{e}\right)^{2k}=\exp\left(2k\ln k+2k\ln\frac{2}{e}\right).
\qe

Let $R$ be a positive constant to be determined later, which serves as the analytic radius of $u(t,x_0)$ at $t_0$. Suppose that $t\in (t_0-R,t_0+R)\cap [0,T]$, then we have
\be\label{eq:anal-rad-cond}
|t-t_0|^{2k}<R^{2k}=\exp(2k\ln R).
\qe

By (\ref{k-th-lap-est}), (\ref{eq:stir-app}), (\ref{eq:anal-rad-cond}) and setting $\varepsilon=2-\alpha-A_1 \ge 0$, we have
\begin{equation}\nonumber\begin{aligned}
|R_{2k-1}|
=&\frac{|\Delta^k u(s,x_0)|}{(2k)!}|t-t_0|^{2k}\\
\leq& \exp[(2-\varepsilon)k\ln k+k\ln(2D)+o(k)\\
&-(2k\ln k+2k\ln\frac{2}{e})+2k\ln R]\\
=&\exp\left[-\varepsilon k\ln k+2k\ln\left(e\sqrt{\frac{D}{2}}R\right)+o(k)\right]\\
=:&\exp[\varphi(k)],
\end{aligned}\end{equation}
where $\varphi(k)=-\varepsilon k\ln k+2k\ln\left(e\sqrt{\frac{D}{2}}R\right)+o(k)$.

If $\varepsilon=2-\alpha-A_1>0$, then for any $R>0$, we have
\[
\lim_{k \to \infty}\varphi(k)=-\infty
\]
and
\be\label{evencase1}
\lim_{k \to \infty}|R_{2k-1}| \leq \lim_{k \to \infty}\exp[\varphi(k)]=0.
\qe

If $\varepsilon=2-\alpha-A_1=0$, let $R<\frac{1}{e}\sqrt \frac{2}{D}$, then
\[
\lim_{k \to \infty}\varphi(k)=-\infty
\]
and
\be\label{evencase2}
\lim_{k \to \infty}|R_{2k-1}| \leq \lim_{k \to \infty}\exp[\varphi(k)]=0.
\qe

Case II: If $n$ is even, i.e. $n=2k-2$, $k\ge 2$, then
$$R_n=R_{2k-2}=\frac{\partial_t^{2k-1}u(s,x_0)}{(2k-1)!}(t-t_0)^{2k-1}.$$
By Lemma \ref{lemma:dev-est} and (\ref{k-th-lap-est}), we have
\begin{equation}\nonumber\begin{aligned}
|\partial_t^{2k-1}u(s,x_0)| \leq & \frac{2}{T} \max \limits_{s \in [0,T]} |\partial_t^{2k-2}u(s,x_0)|+T\max \limits_{s \in [0,T]} |\partial_t^{2k}u(s,x_0)|\\
=& \frac{2}{T} \max \limits_{s \in [0,T]} |\Delta^{k-1}u(s,x_0)|+T\max \limits_{s \in [0,T]} |\Delta^{k}u(s,x_0)|\\
\leq & \frac{2}{T}\exp[(A_1+\alpha)(k-1)\ln(k-1)+(k-1)\ln(2D)+o(k)]\\
&+T\exp[(A_1+\alpha)k\ln k+k\ln(2D)+o(k)]\\
=&\exp[(A_1+\alpha)k\ln k+ k\ln(2D)+o(k)].
\end{aligned}\end{equation}
Combining this with (\ref{eq:stir-app}) and (\ref{eq:anal-rad-cond}), we have 
\begin{equation}\nonumber\begin{aligned}
|R_{2k-2}|
=&\frac{|\partial_t^{2k-1}u(s,x_0)|}{(2k-1)!}|t-t_0|^{2k-1}\\
\leq&\frac{2k}{R} \frac{|\partial_t^{2k-1}u(s,x_0)|}{(2k)!} R^{2k}\\
\leq& \exp[(A_1+\alpha)k\ln k+k\ln(2D)+o(k)\\
&-(2k\ln k+2k\ln\frac{2}{e})+2k\ln R+\ln\frac{2k}{R}]\\
=&\exp\left[-\varepsilon k\ln k+2k\ln\left(e\sqrt{\frac{D}{2}}R\right)+o(k)\right].
\end{aligned}\end{equation}
By the same argument in Case I, we have the following:\\
If $\varepsilon=2-\alpha-A_1>0$, then for any $R>0$,
\be\label{oddcase1}
\lim_{k \to \infty}|R_{2k-1}|=0.
\qe
If $\varepsilon=2-\alpha-A_1=0$, let $R<\frac{1}{e}\sqrt \frac{2}{D}$, then
\be\label{oddcase2}
\lim_{k \to \infty}|R_{2k-1}| =0.
\qe

By (\ref{evencase1}) and (\ref{oddcase1}), we prove that $u$ is analytic with analytic radius $r=+\infty$ if $A_1<2-\alpha$. By (\ref{evencase2}) and (\ref{oddcase2}), we prove that $u$ is analytic with analytic radius $r\geq \frac{1}{e}\sqrt \frac{2}{D}$ if $A_1=2-\alpha$. This completes the proof of the theorem.

\end{proof}

Next we prove Theorem~\ref{thm:main1}.
\begin{proof}[Proof of Theorem~\ref{thm:main1}]
Set $w=u-v.$ Then $w$ is the solution of the homogeneous wave equation on $(-T,T)\times V.$ By Theorem~\ref{main-result}, $w(t,x)$ is analytic in $(-T,T)$ for any $x\in V.$ Note that $w(0,x)=\partial_tw(0,x)=0$ for all $x\in V.$ Since $\partial_t^{2k} u = \Delta^k u,$ $k\in \N,$ $$\partial_t^{2k}u(0,x)=0,\quad \forall x\in V.$$ Moreover,
$$\partial_t^{2k-1}u(0,x)=\Delta^{k-1} (\partial_tu)(0,x)=0,\quad \forall x\in V.$$ Since $w(t,x)$ is analytic, $w\equiv 0.$ This proves the theorem.
\end{proof}


\section{Linear evolution equations on a discrete set}\label{sec:derive}

In this section, we study the uniqueness class of solutions to a class of linear evolution equations on a discrete set.

\subsection{Linear operators on a discrete set and the induced weighted directed graph}
Firstly, we recall the setting of directed graphs. A \emph{directed graph} or \emph{digraph} is an ordered pair $G=(V,E)$, where $V$ is the vertex set and $E$ is the directed edge set, which is a set of ordered pair of elements in $V$. We allow self-loops in this section, i.e. $(x,x)\in E$ for some $x\in V$. Let 
$$\deg^+(x)=\# \{y\in V: (x,y) \in E\}$$
be the out-degree, and 
$$\deg^-(x)=\# \{y\in V: (y,x) \in E\}$$
be the in-degree of $x\in V$. We say $(V,E)$ is locally finite if $\deg^+(x)<\infty$ for all $x\in V$.
We write $x\leadsto y$ if $(x,y) \in E$ and let 
$$d(x,y):=\inf\{n:x=z_0\leadsto z_1 \leadsto \cdots \leadsto z_n=y\}$$ 
be the \emph{directed distance} from $x$ to $y$. Though $d$ is generally not a metric of $(V,E)$ as $d(x,y) \ne d(y,x)$ in general, we still have the triangle inequality, i.e. $d(x,y) \le d(x,z)+d(z,y)$. Let 
$$B_R(x):=\{y\in V: d(x,y) \le R\}$$
be the $R$-ball centered at $x$. We say a pair $(G,p)$ is a \emph{connected rooted digraph} with root $p\in V$ if $d(p,x)<\infty$ for all $x\in V$, i.e. there exists a directed path from $p$ to every vertex in $V$.

Given a directed graph $G=(V,E)$, let 
$$q: E \to \R,\quad (x,y)\mapsto q(x,y),$$ 
be the edge weight function. For convenience we extend the edge weight function to $V\times V\to \R$ with $q(x,y)=0$ if $(x,y)\notin E$. We call the triple $G=(V,E,q)$ a \emph{weighted directed graph} or \emph{weighted digraph}. The (weighted) out-Degree is defined via
$$\Deg^+(x)=\sum_{(x,y)\in E}|q(x,y)|,$$
and the (weighted) in-Degree is defined via
$$\Deg^-(x)=\sum_{(y,x)\in E}|q(y,x)|.$$

Next, we consider the linear operators on the function space on a discrete space. Let $V$ be a discrete set and $\R^V$ be the set of all real functions on $V$. Let
$$L:\R^V \to \R^V$$ 
be a linear operator.
We say $L$ is \emph{finitely-determined} if for any $x\in V$ there exists a finite subset $V_x\subset V$, such that $Lu(x)=0$ holds for all $u\in \R^V$ satisfying $u|_{V_x}=0$. Then for any $u\in \R^V$, we have
$$Lu(x)=\sum_{y\in V_x}L\mathbbm{1}_y(x)u(y),$$
where $\mathbbm{1}_y$ is the characteristic function of $y\in V$ defined via
$$
\mathbbm{1}_y(x)=
\begin{cases}
1,\quad x=y;\\
0, \quad x\ne y.
\end{cases}
$$
Given a finitely-determined linear operator $L$, the finite subset $V_x$ can be chosen in the following way: 
$$V_x=S_L^x:=\{y\in V: L\mathbbm{1}_y(x)\ne 0\}.$$
Let $\mathcal{L}(V)$ be the set of all finitely-determined linear operators on $\R^V$, which is the operator class we will study in this section. In the following, we will see that there is a natural connection between $\mathcal{L}(V)$ and all the locally finite weighted digraphs with vertex set $V$.

\begin{definition}[Digraph induced by a linear operator]
Given a discrete set $V$ and a linear operator $L\in \mathcal{L}(V)$, we say that $G(L)=(V,E,q)$ is the digraph induced by $L$ if 
$$q(x,y)=L\mathbbm{1}_y(x),$$
and $$E=\{(x,y):\ q(x,y) \ne 0\}.$$
\end{definition}

One easily verifies that the induced graph $G(L)$ is locally finite for all $L\in \mathcal{L}(V)$. Conversely, given a locally finite weighted digraph $G=(V,E,q)$, we can also define a linear operator $L\in \mathcal{L}(V)$ satisfying $S_L^x=\{y:(x,y) \in E\}$ via 
$$Lu(x)=\sum_{y\in V} q(x,y)u(y).$$ 

\begin{example}[The left shift operator on $\R^{\N}$] Since an infinite sequence can be regarded as a function on $\N$, we denote by $\R^{\N}$ the set of all infinite sequences. Let $L$ be the left shift operator, i.e. 
$$L:(a_1,a_2,a_3,\cdots)\mapsto (a_2,a_3,\cdots).$$
Then $L\in \mathcal{L}(\N)$, and $q(i,j)=1$ if $j=i+1$, otherwise $q(i,j)=0$. 

\begin{figure}[H]
\centering
\includegraphics[width=0.65\textwidth]{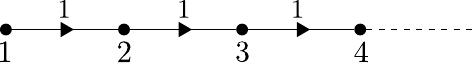}
\caption{The digraph induced by the left shift operator.}
\label{fig.halfline}
\end{figure}

\end{example}
As an analogy of Lemma \ref{lemma:lap-est}, we have the following proposition.
\begin{prop}\label{prop:norm-est} Let $L\in \mathcal{L}(V)$, $u\in \R^V$, and $G=(V,E,q)$ be the digraph induced by $L$, then for any positive integer $k$, we have
$$|L^k u(x)|\leq \left( \sup\limits_{y \in B_k (x)}\Deg^+ (y)\right)^k\sup\limits_{z\in B_k (x)}|u(z)|.$$
\end{prop}

\subsection{Uniqueness of solutions to linear evolution equations on a discrete set}
Consider the following evolution equation on a discrete set $V$:
\begin{align}\label{eq:evol}
\left\{\begin{array}{lr}
\partial_{t}^{m} u(t, x) - Lu(t, x)= f(t,x), & (t, x) \in(-T, T) \times V, \\
u(0, x)=g(x), & x \in V, \\
\partial_t u(0,x)=g_1(x), & x \in V,\\
\qquad \cdots & \cdots\\
\partial_t^{(m-1)}u(0, x)=g_{m-1}(x), & x \in V. \\
\end{array}\right.
\end{align}

By Proposition \ref{prop:norm-est} and Theorem \ref{ore-dev-est}, we have the following results.

\begin{theorem}\label{Thm:evolution-1}
Let $V$ be a discrete set, $L\in \mathcal{L}(V)$, $G=(V,E,q)$ be the digraph induced by $L$, and $C,D$ be positive constants. Suppose $(G,p)$ is a connected rooted digraph and the out-Degree satisfies 
\begin{equation*}
    \Deg^+(x)\le Dd(p,x)^{\alpha},\quad 0\le \alpha \le m, \quad x\ne p.
\end{equation*}
If $u$ is a solution of (\ref{eq:evol}) with $f(t,x)\equiv 0$ and satisfying
\begin{equation*}
    |u(t,x)|\le Cd(p,x)^{(m-\alpha)d(p,x)},\quad x\ne p,
\end{equation*}
then $u$ is time-analytic.
\end{theorem}

\begin{theorem}\label{Thm:evolution-2}
Let $V$ be a discrete set, $L\in \mathcal{L}(V)$, $G=(V,E,q)$ be the digraph induced by $L$, and $C_1, C_2, D$ be positive constants. Suppose $(G,p)$ is a connected rooted digraph and the out-Degree satisfies 
\begin{equation*}
    \Deg^+(x)\le Dd(p,x)^{\alpha},\quad 0\le \alpha \le m, \quad x\ne p.
\end{equation*}
If $u,v$ are solutions of (\ref{eq:evol}) and satisfying
\begin{equation*}
    |u(t,x)|\le C_1d(p,x)^{(m-\alpha)d(p,x)}, \quad x\ne p,
\end{equation*}
and 
\begin{equation*}
    |v(t,x)|\le C_2d(p,x)^{(m-\alpha)d(p,x)}, \quad x\ne p,
\end{equation*}
then $u\equiv v$.
\end{theorem}

The proofs are similar to those of Theorem \ref{main-result} and Theorem \ref{thm:main1}, so we omit them. By the above results, we can deal with the uniqueness of solutions to a wide class of equations. Here are some examples.

\begin{example}[Laplacian on weighted graphs]
Let $G=(V,E,\mu,\omega)$ be an undirected weighted graph, $L=\Delta \in \mathcal{L}(V)$ be the graph Laplacian on $G$. Let $G'=(V',E',q)$ be the weighted digraph induced by $\Delta$, then we have 
\begin{equation*}
    q(x,y)=
    \begin{cases}
    \frac{\omega(x,y)}{\mu(x)}, \qquad &x\ne y, \ x\sim y \text{ in } G;\\
    -\Deg_G(x), &x=y;\\
    0, &\text{otherwise};
    \end{cases}
\end{equation*}
and 
$$E'=\{(x,y):\{x,y\} \in E\} \cup \{(x,x):x \in V \}.$$
Moreover, we have $d_{G'}(x,y)=d_{G'}(y,x)=d_G(x,y)$ for all $x,y \in V$, $d_{G'}$ is a metric of $G'$, and $G'$ is isometric to $G$ as a metric space. Consider the out-Degree on $G'$, we have $\Deg^+_{G'}(x)=2\Deg_G(x)$. Then by Theorem \ref{Thm:evolution-1} and Theorem \ref{Thm:evolution-2}, we can handle the uniqueness class of solutions to equations of the following form:
$$\partial_{t}^{m} u(t, x) - \Delta u(t, x)= f(t,x).$$
For $m=2$, this is exactly what we did in Theorem \ref{thm:main1}, and for $m=1$, we refer to \cite{han2019time}. \\
This approach also applies to the Schr\"odinger operator $\Delta+W(x)$ on undirected or directed weighted graphs, where $W(x)\in \R^V$ is a potential function on $V$.
\end{example}

\begin{example}[The $n$-th power of $L\in \mathcal{L}(V)$]
Let $V$ be a discrete set and $L\in \mathcal{L}(V)$. Then $L^2=L\cdot L \in \mathcal{L}(V)$. Let $G=(V,E,q)$ be the digraph induced by $L$. We define a digraph $G'=(V,E',q')$ via
$$q'(x,y)=\sum_{z\in V}q(x,z)q(z,y),$$
and 
$$E'=\{(x,y): d(x,y)\le 2\}.$$
Note that $G'$ may be different from the digraph $G(L^2)$ induced by $L^2$. Consider $G'$, we have
$$d_{G'}(x,y)=\left \lceil \frac{d_G(x,y)}{2} \right \rceil,$$
and $$\Deg_{G'}^+(x)=\sum_{y}|q'(x,y)|\le \max_{y\in B_1(x) \text{ in } G}(\Deg_G^+(y))^2.$$
One also easily verifies that $L^n \in \mathcal{L}$ for any positive integer $n$, then by Theorem \ref{Thm:evolution-1} and Theorem \ref{Thm:evolution-2}, we can handle the uniqueness class of solutions to equations of the following form:
$$\partial_{t}^{m} u(t, x) - L^nu(t, x)= f(t,x). $$
Let $L=\Delta$ be the Laplacian on graphs and $n=2$, $L^2=\Delta^2$ is called the biharmonic operator.
\end{example}

The following figure shows the digraphs induced by $\Delta$ and $\Delta^2$ on the infinite line graph $\Z$, which is defined in Section \ref{sec:sharp}.
\begin{figure}[H]
\centering
\subfigure[Digraph induced by $\Delta$ on $\Z$, $\Deg^+(x) \equiv 4$.]{
\includegraphics[width=0.75\textwidth]{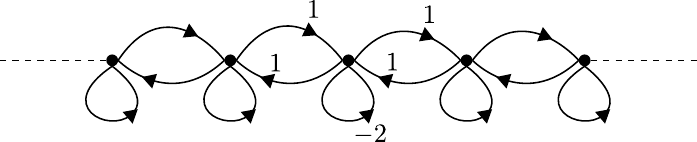}}
\qquad 
\subfigure[Digraph induced by $\Delta^2$ on $\Z$, $\Deg^+(x) \equiv 16$.]{
\includegraphics[width=0.75\textwidth]{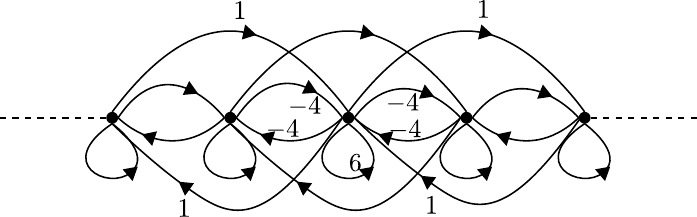}}

\caption{}
\label{fig.line}
\end{figure}

\subsection{For vector-valued function spaces} Given a set $V$, let 
$$\R^{NV}:=\{u|\  u:V \to \R^N\}$$
be the set of all $N$-dimensional vector-valued functions on $V$. Similarly with the previous discussion, we can also define finitely-determined linear operators on $\R^{NV}$, and we still use $\mathcal{L}(V)$ to represent the set of these operators. Given a linear operator $L\in \mathcal{L}(V)$, let
$$ S_L^x:=\{y\in V : \ L\mathbbm{1}^{e_i}_y(x) \ne \ 0 \text{ for some } i,\ 1\le i \le N\},$$
where $e_i=(0,\cdots, 1, \cdots, 0)^T$ is the $i$-th unit vector and
$$
\mathbbm{1}_y^{e_i}(x)=
\begin{cases}
e_i,\qquad & x=y;\\
0, & x\ne y.
\end{cases}
$$
Then we have
$$Lu(x)=\sum_{y\in S_L^x}Q(x,y)u(y) \in \R^N,$$
where $Q(x,y)$ is a $N\times N$ real matrix with $Q_{i,j}(x,y)=L\mathbbm{1}_y^{e_i}(x)\cdot e_j$.

Similarly, we can also define $G(L)=(V,E,Q)$ be the digraph induced by $L$, which is a digraph with matrix-valued edge weight $Q$. We define the (weighted) out-Degree via
$$\Deg^+(x):=\sum_{y\in S_L^x}\sum_{i,j}|Q_{i,j}(x,y)|.$$
Given a vector $\nu=(\nu_1,\nu_2,\cdots,\nu_N)$, we denote by $||\nu||_{\infty}=\max|\nu_i|$ the infinity norm of $\nu$. Then for all $u\in \R^{NV}$, we have
$$||Lu(x)||_{\infty}\le \Deg^+(x)\max_{y\in S_L^x}||u(y)||_{\infty},$$
and 
$$||L^k u(x)||_{\infty}\leq \left( \sup\limits_{y \in B_k (x)}\Deg^+ (y)\right)^k\sup\limits_{z\in B_k (x)}||u(z)||_{\infty}.$$

Results analogous to Theorem \ref{Thm:evolution-1} and Theorem \ref{Thm:evolution-2} can also be derived. The following are two examples of linear operators on the space of vector-valued functions on graphs.

\begin{example}[Laplacian on graphs with complex-valued weight]
Given a undirected graph $G=(V,E,\mu,\omega)$ with complex-valued weight, where $\omega:E\to \C$ is the edge weight function and $\mu:V\to \C \setminus\{0\}$ is the vertex weight function. Consider the Laplacian on the complex-valued function space $\C^V$
$$\Delta u(x):=\frac{1}{\mu(x)}\sum_{y,y\sim x}\omega(x,y)(u(y)-u(x)).$$
We write the multiplication of complex numbers
 $(a+ib)(c+id)=(ac-bd)+i(ad+bc)$ in the following way:
\begin{equation*}
\begin{pmatrix}
a & -b \\
b & a
\end{pmatrix}
\begin{pmatrix}
c \\
d
\end{pmatrix}
=\begin{pmatrix}
ac-bd \\
ad+bc
\end{pmatrix}.
\end{equation*}
Then $\Delta \in \mathcal{L}(V)$ and $d_{G'}(x,y)=d_G(x,y)$, where $G'=(V,E',Q)$ is the digraph induced by $\Delta$, and $G'$ is with edge weight
\begin{equation*}
Q(x,y)=
\begin{cases}
\frac{1}{\mu(x)}\omega(x,y), \qquad & \text{$x\sim y$ in $G$}; \\
-\Deg_G(x), & x=y;\\
0, &\text{otherwise.}
\end{cases}
\end{equation*}
\end{example}

\begin{example}[Connection Laplacian on graphs]
Given a weighted undirected graph $G=(V,E,\mu,\omega)$, we assign a signature $\sigma_{xy}$ to each oriented edge $(x,y)$ with $\{x,y\} \in E$, where each $\sigma_{xy}$ is a $N$-dimensional orthogonal matrix satisfying $\sigma_{xy}=\sigma_{yx}^{-1}$. Or we write the signature as a map $\sigma: E^{or} \to O(N)$, where $E^{or}=\{(x,y):\{x,y\} \in E\}$ is the set of oriented edges and $O(N)$ is the $N$-dimensional orthogonal group. For any vector-valued function $u:V \to \R^N$, the connection Laplacian is defined via 
$$\Delta^{\sigma}u(x):=\frac{1}{\mu(x)}\sum_{y,y\sim x} \omega_{xy}(\sigma_{xy}u(y)-u(x)) \in \R^N. $$
We refer to \cite{liu2019curvature} for more on connection Laplacian on graphs.\\
Let $G'=(V,E',Q)$ be the digraph induced by the connection Laplacian $\Delta^{\sigma}$. Then $d_{G'}(x,y)=d_{G}(x,y)$ and 
\begin{equation*}
Q(x,y)=
\begin{cases}
\frac{1}{\mu(x)}\omega(x,y)\sigma_{xy},\qquad & \text{$x\sim y$ in $G$}; \\
-\Deg_G(x)I_N, & x=y;\\
0, &\text{otherwise.}
\end{cases}
\end{equation*}

\end{example}

\section*{Acknowledgments}

B.H. is supported by NSFC, no.11831004 and no.11926313.

\bibliographystyle{alpha}
\bibliography{Bib-wave}

\end{document}